\documentclass[11pt,a4paper]{amsart}
\pdfoutput=1
\usepackage[applemac]{inputenc}
\usepackage[T1]{fontenc}
\usepackage[english]{babel}
\usepackage{
	amssymb,
	mathtools,
	todonotes,
	microtype
}
\usepackage{geometry}
\newtheoremstyle{lemma}{.5\baselineskip\@plus.2\baselineskip\@minus.2\baselineskip}{.5\baselineskip\@plus.2\baselineskip\@minus.2\baselineskip}
	{\itshape}
	{}
	{\bfseries}
	{.}
	{\newline}
	{\thmname{#1}\thmnumber{ #2}\thmnote{ (#3)}}	
\theoremstyle{lemma}
	\newtheorem{theorem}{Theorem}
	\newtheorem{lemma}[theorem]{Lemma}  
	
	\newtheorem{corollary}[theorem]{Corollary}  

\newtheoremstyle{definition}{.5\baselineskip\@plus.2\baselineskip\@minus.2\baselineskip}{.5\baselineskip\@plus.2\baselineskip\@minus.2\baselineskip}
	{}
	{}
	{\bfseries}
	{.}
	{\newline}
	{\thmname{#1}\thmnumber{ #2}\thmnote{ (#3)}}	
\theoremstyle{definition}

\def\XXint#1#2#3{{\setbox0=\hbox{$#1{#2#3}{\int}$} 
\vcenter{\hbox{$#2#3$}}\kern-.5\wd0}}

\makeatletter
\newcommand\avsuminner[2]{%
  {\sbox0{$\m@th#1\sum$}%
   \vphantom{\usebox0}%
   \ooalign{%
     \hidewidth
     \smash{\vrule height\dimexpr\ht0+1pt\relax depth\dimexpr\dp0+1pt\relax}%
     \hidewidth\cr
     $\m@th#1\sum$\cr
   }%
  }%
}
\makeatother

\newcommand{\N}{\ensuremath{\mathbb{N}}}

\newcommand{\R}{\ensuremath{\mathbb{R}}}

\newcommand{\HM}{\ensuremath{\mathcal{H}}}

\DeclarePairedDelimiter\abs{\lvert}{\rvert}
\DeclarePairedDelimiter\norm{\lVert}{\rVert}
\newcommand{\dd}{\ensuremath{\mathrm{d}}}

\renewcommand{\phi}{\varphi}
\renewcommand{\epsilon}{\varepsilon}

\usepackage{hyperref}
\hypersetup{	
	colorlinks,
	linkcolor=black,
	citecolor=black,
	urlcolor=black,
	breaklinks=true
}

\hypersetup{
	pdftitle={Comparing maximal mean values on different scales},
	pdfsubject={},
	pdfauthor={Thomas Havenith and Sebastian Scholtes},
	pdfkeywords={},
	pdfcreator={},
	pdfproducer={}
}

\makeindex
\begin{document}

\title{Comparing maximal mean values on different scales}
\author{\href{mailto:thomas.havenith@deutschebahn.com}{Thomas Havenith}}
\address{\flushleft Thomas Havenith, Infrastrukturplanung (I.ETZ 1), DB Energie GmbH, Pfarrer-Perabo-Platz 2, 60326 Frankfurt am Main, Germany}
\email{\href{mailto:thomas.havenith@deutschebahn.com}{thomas.havenith@deutschebahn.com}}
\urladdr{\href{http://www.dbenergie.de/}{http://www.dbenergie.de/}}
\author{\href{mailto:sebastian.scholtes@rwth-aachen.de}{Sebastian Scholtes}}
\address{\flushleft Sebastian Scholtes, Institut f\"ur Mathematik, RWTH Aachen University, Templergraben 55, 52062 Aachen, Germany}
\email{\href{mailto:sebastian.scholtes@rwth-aachen.de}{sebastian.scholtes@rwth-aachen.de}}
\urladdr{\href{http://www.instmath.rwth-aachen.de/~scholtes/home/}{http://www.instmath.rwth-aachen.de/~scholtes/home/}}
\date{\today}
\begin{abstract}
	When computing the average speed of a car over different time periods from given GPS data, it is conventional wisdom that the maximal average speed over all time intervals of fixed length 
	decreases if the interval length increases. However, this intuition is wrong. This can be easily seen by the example in Figure \ref{counterexample} for the two
	time periods $T<S$. We investigate this phenomenon and make rigorous in which sense this intuition is still true.
\end{abstract}
\maketitle

\begin{figure}[h]
		\centering 
		\includegraphics[width=0.5\textwidth]{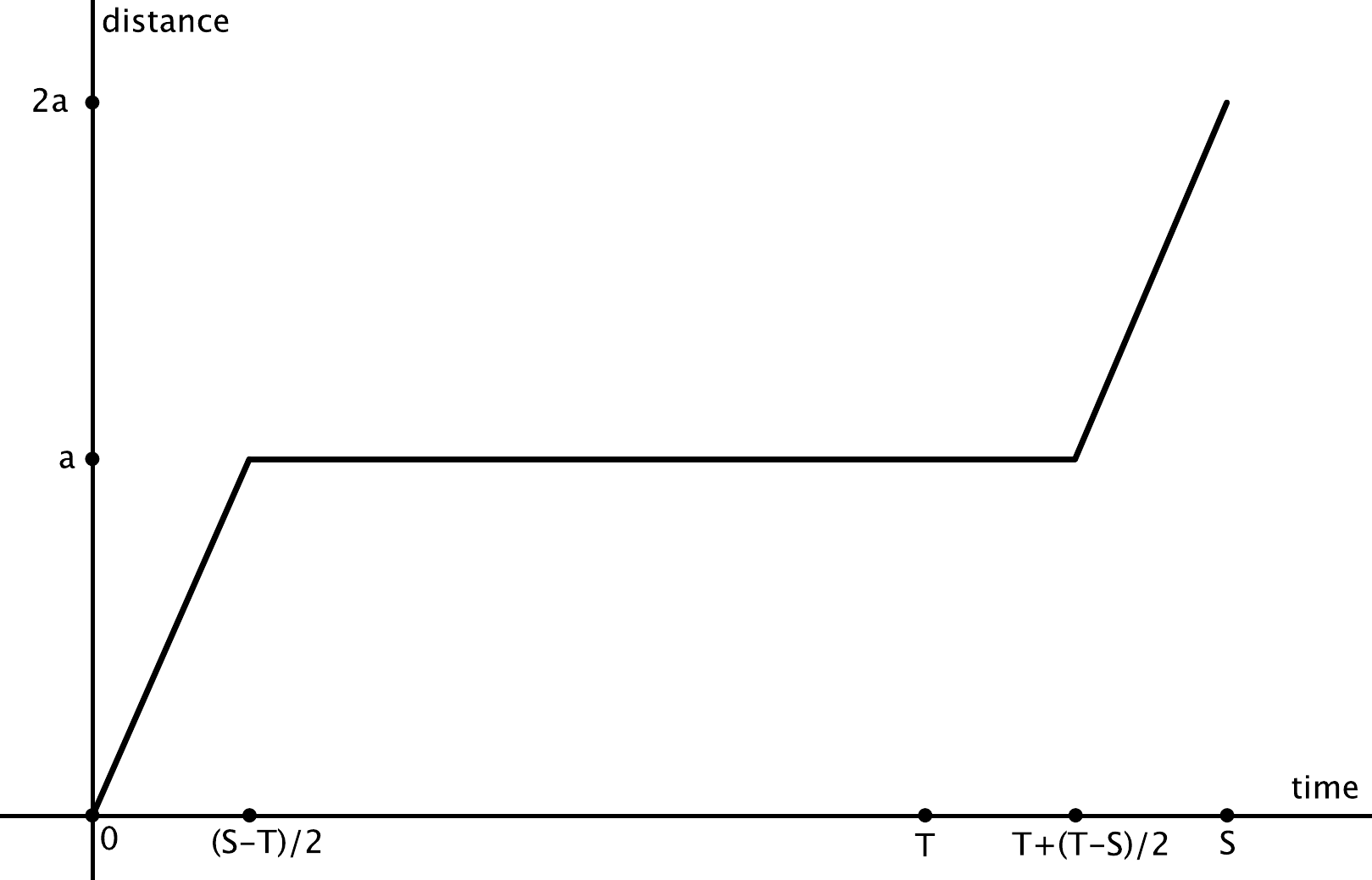}
		\caption{Maximal average speeds for larger time periods are not necessarily smaller.}
		\label{counterexample}
	\end{figure}

In many applications, rates
are measured as average velocities; for instance the speed monitored by Speed Check Services SPECS in the UK, 
capacities of rides in amusement parks
and the drift velocity of electrons are obtained in this way \cite{Smith2012a,Ahmadi1997a,Griffiths1999a}. 
In other fields mean values over different time intervals and their relation are important:
\begin{itemize}
	\item
		An athlete has a high heart rate during a competition and maybe even a very high heart rate for a small time during this contest, but having a similar heart rate for a whole week
		would result in serious health problems.
	\item
		While approximately 10 percent of people that are exposed to a radiation dose of 1 Sv over a short time period die from light radiation poisoning
		in the first 30 days \cite[p.400]{Harper2011a}, this is also the maximal dose allowed for NASA astronauts over their career, \cite[p.7]{NASA2008a}.
	\item
		Certain materials, for example a power line, can endure fairly high stresses, currents etc. for a small time interval, but the average load on a larger time scale has to be much smaller. Another illustration for the effect of the time scale is the phenomenon of creep \cite{Meyers1999a}.
\end{itemize}

All these examples show that one expects maximal mean values to become smaller as the interval size decreases. We study mean values (in the discrete as well as in the continuous setting) 
with regard to this property and find that it is not true in general (Lemmata \ref{normsordered} and \ref{nonordered}), but that it is true if one compares the maximal mean value over all regions of size $T$
with those taken on regions of size $d\cdot T$ for $d\in \N$ (Lemmata \ref{normsordered} and \ref{normsordered2}).

\section{Discrete setting}

Let $\ell^{\infty}$ be the space of bounded real sequences. For $x\in \ell^{\infty}$, $p\in (0,\infty)$ and $n\in\N$ we define
\begin{align*}
	\norm{x}_{\infty,p,n}\vcentcolon =\sup_{j\in\N}\Big(\frac{1}{n}\sum_{i=j}^{j+(n-1)}\abs{x_{i}}^{p}\Big)^{\frac{1}{p}}.
\end{align*}

\begin{lemma}[For $p\in [1,\infty)$ the mapping $\norm{\cdot}_{\infty,p,n}$ is an equivalent norm to $\norm{\cdot}_{\infty}$]
	Let $p\in [1,\infty)$, $n\in\N$. Then $\norm{\cdot}_{\infty,p,n}$ is a norm on $\ell^{\infty}$ with
	\begin{align*}
		n^{-\frac{1}{p}}\norm{x}_{\infty}\leq \norm{x}_{\infty,p,n}\leq \norm{x}_{\infty}.
	\end{align*}
\end{lemma}
\begin{proof}
	Clearly $\norm{\cdot}_{\infty,p,n}$ is positively homogenous of degree one and $\norm{x}_{\infty,p,n}\geq 0$ with equality if and only if $x=0$. From the triangle inequality of the
	$p$-norm in Euclidean $n$-space we obtain the triangle inequality for $\norm{\cdot}_{\infty,p,n}$. Additionally, we obtain
	\begin{align*}
		n^{-\frac{1}{p}}\norm{x}_{\infty}=\Big(\frac{1}{n}\norm{x}_{\infty}^{p}\Big)^{\frac{1}{p}}
		\leq \norm{x}_{\infty,p,n} =\sup_{j\in\N}\Big(\frac{1}{n}\sum_{i=j}^{j+(n-1)}\abs{x_{i}}^{p}\Big)^{\frac{1}{p}}
		\leq \Big(\frac{1}{n}n\norm{x}_{\infty}^{p}\Big)^{\frac{1}{p}}=\norm{x}_{\infty}.
	\end{align*}
\end{proof}

This means that the spaces $(\ell^{\infty},\norm{\cdot}_{\infty,p,n})$ are non-reflexive Banach spaces with the same topology as the standard $\ell^{\infty}$.

\begin{lemma}[Inequality for norms]\label{inequality}
	Let $\alpha_{i}\in\N$ with $n=\sum_{l=1}^{d}\alpha_{l}$. Then for all $p\in [1,\infty)$ and $x\in\ell^{\infty}$ we have that 
	\begin{align*}
		\norm{x}_{\infty,p,n}^{p} \leq \sum_{l=1}^{d}\frac{\alpha_{l}}{n}\norm{x}_{\infty,p,\alpha_{l}}^{p}.
	\end{align*}
\end{lemma}
\begin{proof}
	Let $j\in\N$, $\epsilon>0$ such that $\norm{x}_{\infty,p,n}^{p}-\epsilon\leq \frac{1}{n}\sum_{i=j}^{j+(n-1)}x_{i}^{p}$. Then
	\begin{align*}
		\norm{x}_{\infty,p,n}^{p}-\epsilon\leq\frac{1}{n}\sum_{i=j}^{j+(n-1)}x_{i}^{p}
		=\frac{1}{n}\sum_{l=1}^{d}\sum_{i=j+\sum_{k=1}^{l-1}\alpha_{k}}^{j+\sum_{k=1}^{l-1}\alpha_{k}+(\alpha_{l}-1)}x_{i}^{p}
		\leq \frac{1}{n}\sum_{l=1}^{d} \alpha_{l} \norm{x}_{\infty,p,\alpha_{l}}^{p}.
	\end{align*}
\end{proof}

\begin{lemma}[Norms $\norm{\cdot}_{\infty,p,dn}$ of multiples of $n$ are ordered]\label{normsordered}
	Let $m,n \in\N$ and $n\leq m$. 
	\begin{itemize}
		\item
			If $n$ is a factor of $m$, i.e. $m=dn$ for $d\in \N$ then
			\begin{align*}
				\norm{x}_{\infty,p,dn}=\norm{x}_{\infty,p,m}\leq \norm{x}_{\infty,p,n}\quad\text{for all }x\in\ell^{\infty}.
			\end{align*}
		\item
			If $n$ is no factor of $m$ then there are $x,y\in\ell^{\infty}$ such that
			\begin{align*}
				\norm{x}_{\infty,p,n}< \norm{x}_{\infty,p,m}\quad\text{and}\quad
				\norm{y}_{\infty,p,m}< \norm{y}_{\infty,p,n}.
			\end{align*}
	\end{itemize}
\end{lemma}
\begin{proof}
	If we choose all $\alpha_{l}=n$ in Lemma \ref{inequality} we immediately obtain the first item. For the second item 
	define $x(n)\in\ell^{\infty}$ by $x(n)_{i}=1$ if $i\in\{(d-1)n+1\mid d\in\N\}$ and $x(n)_{i}=0$ otherwise. %else for $n\in\N$. 
	Then
	\begin{align*}
		\norm{x(n)}_{\infty,p,n}^{p}=\frac{1}{n}<\frac{d}{dn-1}\leq \frac{d}{(d-1)n+j} =\norm{x(n)}_{\infty,p,(d-1)n+j}^{p}
	\end{align*}
	for $j=1,\ldots,n-1$ and $d\in\N$. Note that the natural numbers of the form $m=(d-1)n+j$ for $j=1,\ldots,n-1$, $d\in\N$ are exactly those which are either smaller than $n$
	or such that $n$ is no factor of $m$. This proves the proposition as we obtain the opposite inequality by interchanging the roles of $m$ and $n$.
\end{proof}

In the following Lemma we show that counterexamples, however, cannot be too bad.

\begin{lemma}[Counterexamples cannot be too bad]
	Let $p\in [1,\infty)$, $n<m$. Then for all $x\in \ell^{\infty}$ we have that
	\begin{align*}
		\norm{x}_{\infty,p,m}^{p}\leq \frac{(\lfloor\frac{m}{n}\rfloor+1)n}{m}\norm{x}_{\infty,p,n}^{p}\leq 2\norm{x}_{\infty,p,n}^{p}.
	\end{align*}
\end{lemma}
\begin{proof}
	Choose $d\vcentcolon=\lfloor\frac{m}{n}\rfloor$ and $\epsilon>0$. Then there is $j\in\N$ such that
	\begin{align*}
		\MoveEqLeft
		\norm{x}_{\infty,p,m}^{p}-\epsilon\leq \frac{1}{m}\sum_{i=j}^{j+(m-1)}\abs{x_{i}}^{p}
		\leq \frac{1}{m}\sum_{i=j}^{j+((d+1)n-1)}\abs{x_{i}}^{p}\\
		&\leq \frac{(d+1)n}{m}\norm{x}_{\infty,p,(d+1)n}^{p}
		\leq \frac{(d+1)n}{m}\norm{x}_{\infty,p,n}^{p}
	\end{align*}
	according to Lemma \ref{normsordered}.
\end{proof}

\section{Continuous setting}

Let $f:\R^{N}\to\R$ be a function. For $V>0$, $p\in (0,\infty)$ we define
\begin{align*}
	\norm{f}_{p,V}\vcentcolon=\sup_{\substack{\Omega\subset \R^{N}\text{ open}\\\Omega\text{ convex}\\\abs{\Omega}=V}}\Big(\frac{1}{\abs{\Omega}}\int_{\Omega}\abs{f(x)}^{p}\,\dd x\Big)^{\frac{1}{p}}.
\end{align*}
Note that this norm bears a resemblance to the Hardy-Littlewood maximal function from harmonic analysis.
If we denote the space of all functions $f:\R^{N}\to \R$ with $f=0$ a.e. by $\mathcal{N}$ we can define
\begin{align*}
	L_{V}^{p}\vcentcolon=\{f:\R^{N}\to \R\mid \norm{f}_{p,V}<\infty\}/\mathcal{N}.
\end{align*}
By a slight abuse of notation we also want to consider $\norm{\cdot}_{p,V}$ as a mapping from $L_{V}^{p}$ to $\R$ in the obvious manner.

\begin{lemma}[For $p\in [1,\infty)$ the mapping $\norm{\cdot}_{p,V}$ is a norm on $L_{V}^{p}$]
	Let $p\in [1,\infty)$, $V>0$. Then $\norm{\cdot}_{p,V}$ is a norm on $L_{V}^{p}$.
\end{lemma}
\begin{proof}
	Clearly $\norm{\cdot}_{p,V}$ is positively homogenous of degree one and $\norm{f}_{p,V}\geq 0$ with equality if and only if $f=0$ a.e.. 
	Now, the triangle inequality for $L^{p}$ proves the triangle inequality for $\norm{\cdot}_{p,V}$.
\end{proof}

\begin{lemma}[Inequality for norms]\label{inequality2}
	Let $\alpha_{i}>0$ with $V=\sum_{l=1}^{d}\alpha_{l}$. Then for all $p\in [1,\infty)$, $V>0$ and $f\in L_{V}^{p}$ holds
	\begin{align*}
		\norm{f}_{p,V}^{p} \leq \sum_{l=1}^{d}\frac{\alpha_{l}}{V}\norm{f}_{p,\alpha_{l}}^{p}.
	\end{align*}
\end{lemma}
\begin{proof}
	\textbf{Step 1}
		Let $\epsilon>0$ and $\Omega\subset \R^{N}$ be open and connected such that $\norm{f}_{p,V}^{p}-\epsilon\leq \frac{1}{\abs{\Omega}}\int_{\Omega}\abs{f(x)}^{p}\,\dd x$.
		According to Step 2 there is a hyperplane that cuts $\Omega$ into two convex sets $\Omega_{1}$ and $\tilde\Omega_{1}$ with $\abs{\Omega_{1}}=\alpha_{1}$ and 
		$\abs{\tilde\Omega_{1}}=\sum_{l=2}^{d}\alpha_{l}$. Now iterating this procedure with $\tilde\Omega_{1}$ we finally obtain $d$ open convex sets $\Omega_{i}$ with $\abs{\Omega_{i}}=\alpha_{i}$
		and $\Omega=\bigcup_{l=1}^{d}\Omega_{i}\cup A$, where $\HM^{N}(A)=0$, as $A$ is contained in a finite union of hyperplanes. Therefore
		\begin{align*}
			\norm{f}_{p,V}^{p}-\epsilon\leq \frac{1}{\abs{\Omega}}\int_{\Omega}\abs{f(x)}^{p}\,\dd x
			=\frac{1}{\abs{\Omega}}\sum_{l=1}^{d}\int_{\Omega_{i}}\abs{f(x)}^{p}\,\dd x
			\leq \frac{1}{\abs{\Omega}}\sum_{l=1}^{d}\alpha_{l}\norm{f}_{p,\alpha_{l}}^{p}.
		\end{align*}
	\textbf{Step 2}
		Let $\Omega$ be an open convex set and $\alpha+\beta=\abs{\Omega}$ for $\alpha,\beta>0$. If we define
		\begin{align*}
			\Omega_{s}\vcentcolon=\{ x\in\Omega \mid x_{1}=s \}
		\end{align*}
		we know by Cavalieri's Principle (a special case of Fubini's Theorem) that
		the function $\phi(t)\vcentcolon=\int_{-\infty}^{t}\abs{\Omega_{s}}\,\dd s$ measures the volume of the set $\Omega$ intersected with the open half-space below (in $e_{1}$ direction) the 
		hyperplane $\{x\in\R^{N}\mid x_{1}=t\}$. Additionally $\lim_{t\to -\infty}\phi(t)=0$, $\lim_{t\to\infty}\phi(t)=\abs{\Omega}$ and $\phi$ is continuous. Therefore the intermediate value
		theorem yields the existence of a $t_{0}$ such that $\Omega_{1}=\bigcup_{t\in (-\infty,t_{0})}\Omega_{t}$ and $\tilde\Omega_{1}=\bigcup_{t\in (t_{0},\infty)}\Omega_{t}$ are open convex sets
		(as intersection of two such sets) and $\abs{\Omega_{1}}=\alpha$ and $\abs{\tilde\Omega_{1}}=\abs{\Omega}-\alpha=\beta$.
\end{proof}

\begin{lemma}[Norms $\norm{\cdot}_{p,d\cdot V}$ of multiples of $V$ are ordered]\label{normsordered2}
	Let $p\in [1,\infty)$, $V>0$ and $d\in\N$. Then
	\begin{align*}
		\norm{f}_{p,d\cdot V}\leq \norm{f}_{p,V}\quad\text{for all }f\in L_{V}^{p}.
	\end{align*}
\end{lemma}
\begin{proof}
	If we chose all $\alpha_{l}=V$ in Lemma \ref{inequality2} we directly obtain the result.
\end{proof}

\begin{lemma}[For $N=1$, $T<S$ and $T$ no factor of $S$ the norms are not ordered]\label{nonordered}
	Let $N=1$, $p\in [1,\infty)$ and $T,S>0$, $T<S$ such that $T$ is no factor of $S$. Then there are $f,g\in L_{T}^{p}$ such that 
	\begin{align*}
		\norm{f}_{p,T}<\norm{f}_{p,S}\quad\text{and}\quad
		\norm{g}_{p,S}<\norm{g}_{p,T}.
	\end{align*}
\end{lemma}
\begin{proof}
	Let $T<S$  such that $T$ is not a factor of $S$, i.e. $\frac{S}{T}\not\in\N$. Write $d\vcentcolon=\lfloor \frac{S}{T} \rfloor$. Then it is possible to choose $\epsilon>0$ such that 
	$(d+1)\epsilon<S-dT$. Let $\psi_{\epsilon}(x)$ be a bump function of $L^{p}$ mass $1$ with support in $(0,\epsilon)$. Then $f(x)=\sum_{i=0}^{d}\psi_{\epsilon}(x-i(T+\epsilon))$ is an $L_{T}^{p}$ function and
	\begin{align*}
		\norm{f}_{p,T}^{p}=\frac{1}{T}<  \frac{d+1}{S}=\norm{f}_{p,S}^{p},
	\end{align*}
	because $\frac{S}{T}<\lfloor \frac{S}{T} \rfloor+1$. On the other hand, $g=\psi_{\epsilon}\in L_{T}^{p}$ yields
	\begin{align*}
		\norm{g}_{p,S}^{p}=\frac{1}{S}<\frac{1}{T}=\norm{g}_{p,T}^{p}.
	\end{align*} 
\end{proof}

\begin{corollary}[Counterexamples cannot be too bad]
	Let $p\in [1,\infty)$, $T<S$. Then for all $f\in L_{S}^{p}$ holds
	\begin{align*}
		\norm{f}_{p,S}^{p}\leq\frac{(\lfloor\frac{S}{T}\rfloor+1)T}{S}\norm{f}_{T}^{p}\leq 2\norm{f}_{T}^{p}.
	\end{align*}
\end{corollary}
\begin{proof}
	Choose $d\vcentcolon= \lfloor\frac{S}{T}\rfloor$ and $\epsilon>0$ such that $\norm{f}_{p,S}^{p}-\epsilon\leq \frac{1}{S}\int_{\Omega}\abs{f(t)}^{p}\,\dd t$. Then we can
	extend $\Omega$ convexly to an open $\tilde\Omega$ with $\Omega\subset \tilde\Omega$ and $\abs{\tilde\Omega}=(d+1)T$. This can for example be done by taking the convex hull $\tilde\Omega_{y}$
	of $\Omega$ and a point $y$ moving along a normal at a fixed point $y\in\partial \Omega$. By an argument similar to that in the proof of Lemma \ref{inequality2} Step 2, we see that $x$ can 
	be chosen in such a way that $\abs{\tilde\Omega_{y}}=(d+1)T$. Now
	\begin{align*}
		\norm{f}_{p,S}^{p}-\epsilon\leq \frac{1}{S}\int_{\Omega}\abs{f(t)}^{p}\,\dd t
		\leq\frac{1}{S}\int_{\tilde\Omega}\abs{f(t)}^{p}\,\dd t
		\leq\frac{(d+1)T}{S}\norm{f}_{(d+1)T}^{p}
		\leq\frac{(d+1)T}{S}\norm{f}_{T}^{p}
	\end{align*}
	by Lemma \ref{normsordered2}.
\end{proof}

\bibliography{/Users/sebastianscholtes/Documents/library.bib}{}
\bibliographystyle{amsalpha}

\end{document}